\documentclass[a4paper]{article}
\usepackage{amsthm,amssymb,amsmath,enumerate,tikz,tikz-cd,graphicx}\usepackage{graphicx,mathtools}
\usetikzlibrary{positioning,arrows,calc,matrix,tqft,through}
\usepackage[T1]{fontenc}
\usepackage[utf8]{inputenc}

\newcommand{\s}{\ensuremath{\mathcal{S}}}
\newcommand{\N}{\ensuremath{\mathbb{N}}}

\newcommand{\Mcal}{\ensuremath{\mathcal{M}}}
\newcommand{\Pcal}{\ensuremath{\mathcal{P}}}
\newcommand{\Rcal}{\ensuremath{\mathcal{R}}}
\newcommand{\Bcal}{\ensuremath{\mathcal{B}}}
\newcommand{\Ccal}{\ensuremath{\mathcal{C}}}
\newcommand{\Fcal}{\ensuremath{\mathcal{F}}}

\newcommand{\Ocal}{\ensuremath{\mathcal{O}}}

\newcommand{\Ncal}{\ensuremath{\mathcal{N}}}
\newcommand{\Vcal}{\ensuremath{\mathcal{V}}}
\newcommand{\Wcal}{\ensuremath{\mathcal{W}}}

\newcommand{\Scal}{\ensuremath{\mathcal{S}}}
\newcommand{\alt}[2]{\left\{\begin{array}{l}#1\\#2\end{array} \right.}

\newcommand{\es}{\emptyset}
\newcommand{\sm}{\smallsetminus}
\newcommand{\sub}{\subseteq}

\newcommand{\Reff}{\Rcal_{\text{eff}}}
\newcommand{\Reffk}{\Rcal_{\text{eff}}^\kappa}

\def\lc{left-con\-nec\-ted}
\def\td{tree-decom\-po\-si\-tion}
\def\ttd{tree of tree-decom\-po\-si\-tions}

\theoremstyle{definition}

\theoremstyle{plain}

\newtheorem{lemma}{Lemma}[section]

\newtheorem{thm}[lemma]{Theorem}
\newtheorem{coro}[lemma]{Corollary}
\newtheorem{remark}[lemma]{Remark}

\newtheorem{prop}[lemma]{Proposition}

\theoremstyle{remark}

\newtheorem{exam}[lemma]{Example}

\newcommand{\comment}[1]{}

\newenvironment{txteq*}
{
	\begin{equation*}
	\begin{minipage}[t]{0.85\textwidth} 
	\em                                
}
{\end{minipage}\end{equation*}\ignorespacesafterend}

\def\?#1{\vadjust{\vbox to 0pt{\vss\vskip-8pt\leftline{%
				\llap{\hbox{\vbox{\pretolerance=-1
							
\doublehyphendemerits=0\finalhyphendemerits=0
							\hsize16truemm\tolerance=10000\small
							\lineskip=0pt\lineskiplimit=0pt
							\rightskip=0pt 
plus16truemm\baselineskip8pt\noindent
							\hskip0pt        
word is never hyphenated!)
							#1\endgraf}\hskip7truemm}}}\vss}}}

\begin{document}
	
	\title{\bf{Canonical trees of tree-decompositions}}
	\author{Johannes Carmesin\\
		University of Birmingham\\Birmingham, UK\\
		\and
		Matthias Hamann\thanks{Supported by the Heisenberg-Programme of the Deutsche 
			Forschungsgemeinschaft (DFG Grant HA 8257/1-1).}\\
		Mathematics Institute, University of Warwick\\ Coventry, UK
		\and
		Babak Miraftab \\
		Department of Mathematics, University of Hamburg \\ Hamburg, Germany
	}
	
	\date{\today}
	
	\maketitle
	\begin{abstract}

We prove that every graph has a canonical tree of \td s that distinguishes all principal
tangles (these include the ends and various kinds of large finite
dense structures) efficiently.

Here `trees of \td s' are a slightly weaker notion than `\td s' but 
much 
more well-behaved than `tree-like metric spaces'. This theorem is best possible in the sense that 
we 
give an example that `trees of \td s' cannot be strengthened to 
`\td s' 
in the above theorem.

This implies results of Dunwoody and Kr\"on as well as of Carmesin, Diestel, Hundertmark and 
Stein. Beyond that for locally finite graphs our result gives for each $k\in\N$ canonical \td s 
that distinguish all $k$-distinguishable ends efficiently.

\end{abstract}
	
	\section{Introduction}

Automorphisms-group invariant \td s\footnote{See Section \ref{sec_Prelim} for a definition.} of 
infinite graphs are applied to study groups via their Cayley graphs (e.\,g.\ Kr\"on 
\cite{CuttingUp-Revisited}, see also Dunwoody and Kr\"on \cite{VertexCuts}, for the proof of 
Stallings' 
Theorem) or other highly symmetric structures; such as \cite{HLMR,miraftab2019splitting}. For 
applications in structural graph theory or matroid theory where canonicity 
does not play a role, we refer the reader to \cite{topo_ends}. 

\vspace{.3cm}

Often it is measured how well a \td\ displays the rough structure of a graph by the 
highly connected substructures it separates. In our context, the most important highly connected 
substructures are the ends (or more generally the tangles, see below). So here we are interested in 
\td s separating lots of ends. 

When do such \td s exist canonically; that is, invariant under the group of 
automorphisms of the graph\footnote{We remark that we actually construct tree-decompositions that are canonical in a stronger sense. Hence we update our definition of canonical later on in the paper.}? The purpose of this paper is to 
answer this question. On one hand, we 
prove a general decomposition result that implies the existence of such \td s in 
special cases. On the other hand, we complement this with many counterexamples to various related 
conjectures, explaining why we believe that this theorem answers this question in a best possible way. 

\vspace{.3cm}

Given a graph $G$, let $\mu(G)$ be the minimal size of a separator of $G$ separating two ends. 
Dunwoody and Kr\"on \cite{VertexCuts} showed under `mild' additional assumptions that every graph 
has a canonical \td\ separating 
any two ends separable by at most $\mu(G)$ vertices. And they also provided an example of a graph 
that does not have a canonical \td\ separating all ends, see Figure \ref{fig_NoTree}. 
We remark that this graph is not locally finite. Still there are such examples for locally finite graphs, 
see Example~\ref{exam_LocFin} below, but they have to be more exotic in the sense that 
they need to have for every number $k$ two ends that cannot be separated by at most $k$ vertices 
(compare Corollary \ref{intro_loc_fin} below).

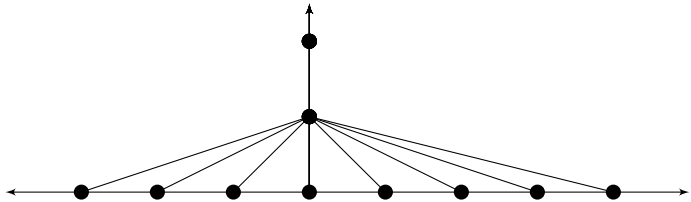
\begin{figure}[htpb]
\centering
\begin{tikzpicture}
\draw[<->,>=latex'] (0,0)--(9,0);
\foreach \i in {1,2,...,8}{
\draw (\i,0) node [circle,fill, inner sep=2pt] {};
\draw (4,1) node [circle,fill, inner sep=2pt] {};
\draw (4,2) node [circle,fill, inner sep=2pt] {};
\draw[->,>=latex'] (4,0)--(4,2.5);
}
\foreach \i in {1,2,...,8}{
\draw (4,1)--(\i,0);}
\end{tikzpicture}
\caption{The depicted graph is obtained from the disjoint union of a ray with a double ray by 
joining the starting vertex of the ray to all vertices on the double ray. This graph has 
no canonical \td\ distinguishing all its ends.}
\label{fig_NoTree}
\end{figure}

\vspace{.3cm}

These two examples may seem counter-intuitive as one might expect to be able to obtain a single 
canonical \td\ by just iterating the result of Dunwoody and Kr\"on as follows. 
Starting with a \td\ of Dunwoody and Kr\"on, we apply the theorem again to each part 
of that \td, and then to each part of that and so on. 
There are some technical aspects to consider when doing this approach, for example one has to 
consider the torsos of the parts and not the parts themselves and one has to take extra care when 
constructing the previous \td s to not spoil a later one. All this put aside for now, 
the above examples show that this plan cannot work.

We have the following perspective on this. 
Intuitively speaking, we define a \emph{tree of \td s} to be a collection of 
all these iterative \td s -- before any sticking together takes place, see Figure \ref{treetd}. The 
main result of this paper is 
that trees of \td s separating all the ends always can be constructed 
canonically.

\begin{figure}[htpb]  
\begin{center}
\includegraphics[height=4cm]{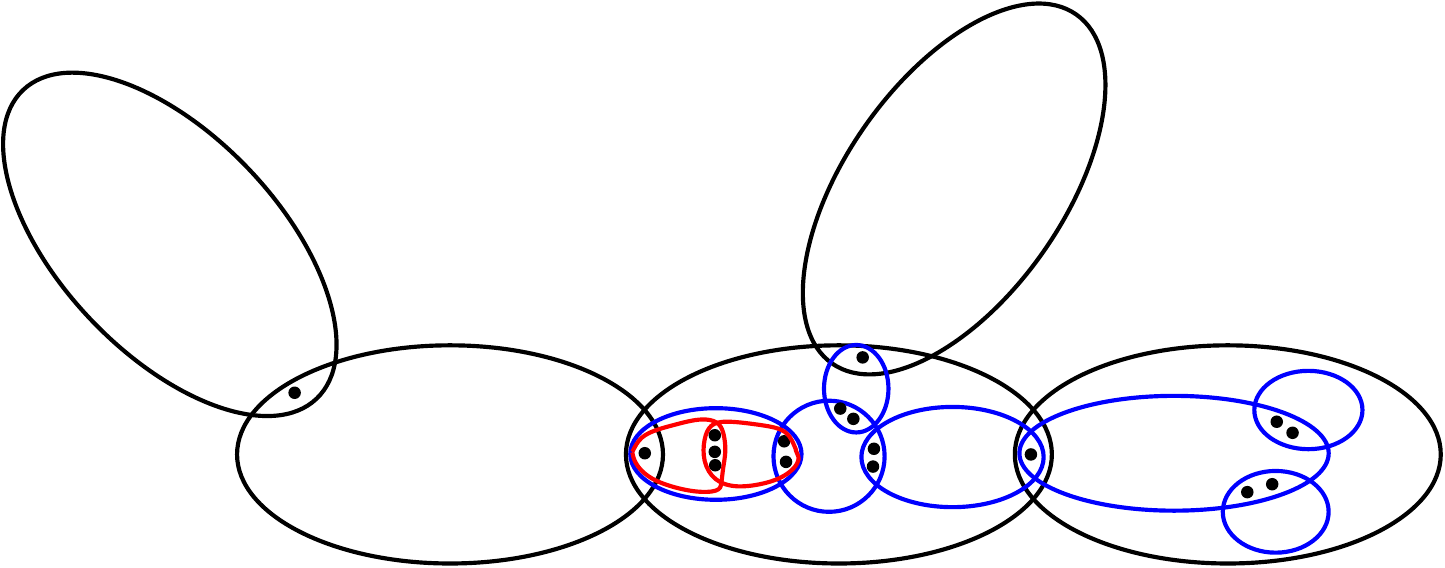}
\caption{A \ttd. Trees of \td s can be thought of as a family of \td s, where each part of one of these \td s is refined by another \td s. For finite graphs, these can be canonically combined into a single \td. The main result of this paper informally says that the tangle tree theorem of Robertson and Seymour can be extended to infinite graphs, in that we get a canonical \ttd\ (that is, a \td\ that is \lq well-defined up to gluing\rq) distinguishing the principal tangles.}\label{treetd}
\end{center}
\end{figure}

Hence the above-mentioned obstructions can only occur in the gluing process from trees of 
\td s to a single \td. More precisely, we can understand the first 
example of Figure \ref{fig_NoTree} as showing that we can not always stick two 
\td s together -- in a canonical way; Example~\ref{exam_LocFin} can be understood to 
show that we can not always stick together infinitely many \td s at once. 

Our main result is the following. Roughly speaking, a \td, or more generally a tree 
of \td s, \emph{distinguishes the set of ends efficiently} if for any two ends it 
contains a separator of minimal size separating the two ends.  
	
	\begin{thm}\label{intro_simple_version}
Every graph has a canonical tree of \td s that distinguishes the set of ends 
efficiently. 
	\end{thm}

While the above mentioned result of Dunwoody and Kr\"on is immediately implied by Theorem 
\ref{intro_simple_version}, our theorem also has the following consequence.

\begin{coro}\label{intro_loc_fin}
 For every $k\in \mathbb{N}$, every locally finite graph has a canonical \td\ that 
distinguishes any two ends distinguishable by at most $k$ vertices efficiently. 
\end{coro}

Tangles were invented to describe dense substructures in finite graphs and play a key role in the 
Graph Minor Theory of Robertson and Seymour \cite{GMX}. It is a 
simple observation that ends of infinite graphs are also tangles, see Section~\ref{sec_TTD} below 
for details. In 
the context of our proof it is technically slightly easier to work with the more general notion of principal
tangles than ends. In fact we prove the following generalisation of Theorem 
\ref{intro_simple_version}.

	\begin{thm}\label{intro_simple2}
Every graph has a canonical tree of \td s that distinguishes the set of principal tangles (or 
even more generally, the set of robust profiles\footnote{See Section \ref{sec_Prelim} for a 
definition of robust 
profiles.}) efficiently. 
	\end{thm}

A side-effect of our general approach is that this theorem also implies recent works concerning 
\td s of finite graphs by Carmesin, Diestel, Hundertmark and Stein~\cite{confing} and by 
 Diestel, Hundertmark and Lemanczyk \cite{DHL-Profiles}.

	\vspace{.3cm}
	
	The remainder of this paper is structured as follows.
	In Section~\ref{sec_Prelim}, we state some main definitions (such as profiles or 
	separations) and prove some basic results on separations of graphs.
	In Section~\ref{sec_TD} we explore the relationship between nested sets of separations and 
\td 
s. Its 
	content is based mainly on Carmesin, Diestel, Hundertmark and Stein~\cite{confing}, other 
parts of the paper are more similar to \cite{topo_ends}.
	Section~\ref{sec_Parts} focuses on how profiles of graphs induce profiles in the torsos of 
its \td s. 
	In Section~\ref{sec_k} we prove our key auxiliary result, which is the existence of nested 
sets of separations distinguishing certain profiles under some mild assumption on the graph.
	In Section~\ref{sec_TTD} we prove our main theorem. After that we discuss the 
connections between profiles and ends, 	$k$-blocks and 	tangles.
	In Section~\ref{sec_LocFin} we deduce our aforementioned results on 
locally finite graphs, and then we have some final remarks.
	
	\section{Sets of separations}\label{sec_Prelim}
	
	For basic notations and terminology for graphs, we refer readers to \cite{DiestelBook17}.
	Let	$G$ be a graph. 
	A \emph{separation} of $G$ is an ordered pair $(A,B)$ of vertex sets such that $G[A]\cup 
G[B]=G$; 
	that is,\ there 
	is no edge between $A\sm B$ and $B\sm A$.
	It is \emph{proper} if neither $(A,B)\leq (B,A)$ nor $(B,A)\leq (A,B)$, that is, if $A\sm 
B\neq\es$ and $B\sm A\neq\es$.
	The \emph{order} of the separation $(A,B)$ is the size of its \emph{separator} $A\cap B$.
	
	For two separations $(A,B)$ and $(C,D)$, we write $(A,B)\leq (C,D)$ if $A\subseteq C$ and 
	$D\subseteq B$.
	It is easy to see that this defines a partial order on the set of all separations.
	A separation $(A,B)$ is \emph{nested} with a separation $(C,D)$, denoted by $(A,B)\parallel 
(C,D)$, 
	if $(A,B)$ is comparable with $(C,D)$ or $(D,C)$.
	If $(A,B)$ is not nested with $(C,D)$ then these two separations \emph{cross}. 
	The \emph{centre} of two separations $(A,B)$ and $(C,D)$ is $A\cap B\cap C\cap D$, their 
four 
	\emph{corners} are the sets $A\cap C$, $B\cap C$, $B\cap D$, and $A\cap D$.
	The corners $A\cap C$ and $B\cap D$ are \emph{opposite} as are the corner $B\cap C$ and 
$A\cap D$.
	Corners are \emph{adjacent} if they are not opposite.
	The \emph{link} between two adjacent corner is their intersection without the centre. 
	For a corner $E\cap F$ with $\{E,E'\}=\{A,B\}$ and $\{F,F'\}=\{C,D\}$, its \emph{interior} 
is 
	$(E\cap F)\sm(E'\cup F')$ and its \emph{corner separation} is the separation $(E\cap 
F,E'\cup F')$.
	We call the corner separations $(E\cap F,E'\cup F')$ and $(E\cup F,E'\cap F')$ 
\emph{opposite}.
	
	A set $P$ of separations of a graph $G$ is a \emph{profile} if it satisfies the following 
	conditions.
	\begin{enumerate}[(P1)]
		\item $P$ is \emph{consistent}, i.\,e.\ if for every two separations $(A,B)$ and 
$(C,D)$ with $(C,D) 
		\leq (A,B)$ and $(A,B)\in P$ we have $(D,C) \notin P$;
		\item if $(A, B), (C,D) \in P$, then $(B\cap D,A\cup C)\notin P$.
	\end{enumerate}
	If two separations do not form a consistent set we say that they \emph{point away} from 
each 
other.
A profile $P$ is \emph{principal} if for every family $((A_i,B_i))_{i\in I}$ in~$P$ with $A_i\cap B_i=A_j\cap B_j$ for all $i,j\in I$, we have $\bigcap_{i\in I} (B_i\sm A_i)\neq\es$.
	A profile $P$ of~$G$ is a \emph{$k$-profile} if all separations in~$P$ have order less than 
$k$ and 
	if for every separation $(A,B)$ of~$G$ of order less than $k$, either $(A,B) \in P$ or 
$(B,A)\in P$.
Note that in a principal $k$-profile $P$ there is for every vertex set $S$ of size less than~$k$ some component $C$ of $G-S$ such that $(V(G)\sm C,C\cup S)\in P$.
	
	In finite graphs, (P2) ensures that profiles cannot hide in small separators.
	But this fails for infinite graphs if we ask only (P2) but not that the profiles are principal as the following example shows.
	
	\begin{exam}
		Let $G$ be the countably infinite star with vertex set $\{c,v_1,v_2,\ldots\}$, 
where 
$c$ is the center of the 
		star.
		Let $\Fcal$ be a non-principal ultrafilter on~$\N$.
		Let $P$ be the set consisting of the separations $(A_F,B_F)$, where 
$B_F:=\{c\}\cup\{v_i\mid i\in 
		F\}$ and $A_F:= (V(G)\sm B_F)\cup\{c\}$, together with the separations $(A,V(G))$ 
with $|A|\leq 1$.
		It follows from the definition of ultrafilters that $P$ is a profile but not principal.
	\end{exam}
	
	A separation $(A,B)$ \emph{distinguishes} two profiles $P$, $P'$ if $(A,B) \in P$ and
	$(B,A)\in P'$ or vice versa and it distinguishes the profiles \emph{efficiently} if there 
is 
no 
	separation of smaller order distinguishing $P$ and $P'$.
	Two profiles are \emph{$\ell$-distinguishable} if a separation $(A,B)$ of order at most 
$\ell$ 
	distinguishes them.
	For a set $\Pcal$ of profiles, let $\kappa(\Pcal,G)$ denote the minimum order of 
separations 
	separating two profiles of~$\Pcal$.
	
	A separation of a graph~$G$ is \emph{relevant} (for some $k\in\N\cup\{\infty\}$ and with 
respect to~$\Pcal$) if it has finite order of at most~$k$ and distinguishes two profiles in~$\Pcal$.
	It is easy to see that relevant separations $(A,B)$ are proper.
	We denote by $\Rcal(k,\Pcal,G)$ the set of relevant separations of~$G$ with respect 
to~$\Pcal$.
	We use $\Rcal(k,\Pcal)$ or $\Rcal(k)$ if $G$ and $\Pcal$ are obvious from the context.
	Let $\Reffk(\Pcal,G)$ be the set of all  separations of order $\kappa(\Pcal,G)$ 
distinguishing two 
	profiles of~$\Pcal$ efficiently and for two profiles $P,P'\in\Pcal$ set 
	\[
	\Reff(P,P'):=\left\{ (A,B)\in\Reffk(\Pcal,G)\mid (A,B)\text{ distinguishes }P,P'\text{ 
		efficiently}\right\}.
	\]

	A component $C$ of $G-(A\cap B)$ for a separation $(A,B)$ of~$G$ is \emph{degenerated} if 
$N(C)$ is a proper subset of $A\cap B$.
	We call a separation \emph{degenerated} relative to $(A,B)$ if it is of the form $(C\cup 
N(C),V(G)\sm C)$ where $C$ is a degenerated component of $G-(A\cap B)$.
	The \emph{degenerator} of a set $\Scal$ of separations is the set of separations that are 
degenerated relative to some separation in~$\Scal$.
	We denote the degenerator of $\Rcal(k,\Pcal,G)$ by $\Scal(k,\Pcal,G)$ and write 
$\Scal(k,\Pcal)$ or $\Scal(k)$ if $G$ and $\Pcal$ are obvious from the context.
	We call $G$ \emph{well-separable} (with respect to a set $\Pcal$ of profiles) if for every 
separation $(A,B)$ of order $\kappa(\Pcal,G)$ no component of $G-(A\cap B)$ is degenerated, that 
is, 
if $\Scal(\kappa(\Pcal,G),\Pcal,G)=\es$.
	A separation $(A,B)$ is \emph{\lc} if $A\sm B$ is connected.
	
	For $n\in\N$, a profile $P$ is \emph{$n$-robust} if for every $(A,B)\in P$ and every 
separation $(C,D)$ of order at most~$n$ 
	the following holds: if $(B\cap C, A\cup D)$ and $(B\cap D, A\cup C)$ both have order less 
than 
	$|A\cap B|$, then one of those corner separations does not lie in~$P$.
	It is \emph{robust} if it $n$-robust for every $n\in\N$.
	
	A separator separates two vertices \emph{minimally} if no other separator of smaller size 
separates them, too.
	The following result is by Halin~\cite{halin91}.
	
	\begin{lemma}\label{lem_minimally}{\rm\cite[2.4]{halin91}}
		Let $G$ be a graph, $u,v\in V(G)$ and $k\in\N$.
		Then there are only finitely many separators of size at most $k$ separating $u$ and 
$v$ 
		minimally.\qed
	\end{lemma}
			
	For a set $\Scal$ of separations we define the property:
	\begin{equation}\tag{$*$}\label{property_Finiteness}
	\begin{minipage}[t]{0.9\textwidth}
	\em
	For all $(A,B),(C,D)\in \Scal$, there are only finitely many $(E,F)\in\Scal$ with $(A,B)< 
(E,F)< 
	(C,D)$.
	\end{minipage}
	\end{equation}
		
	\begin{lemma}\label{lem_make-tree}
		Let $G$ be a graph and let $k\in \N$.
		Then any nested set of \lc\ separations of order at most $k$ has property $(*)$.
	\end{lemma}
	
	\begin{proof}
		Let $\Scal$ be a nested set of \lc\ separations of order at most~$k$ and let 
$(A,B),(C,D)\in\Scal$ with $(A,B)\leq (C,D)$.
		Let us suppose that $\Scal$ does not satisfy property (\ref{property_Finiteness}).
		Then there are infinitely many distinct $(E,F)\in\Scal$ with $(A,B)<(E,F)<(C,D)$.
		Since all of them are \lc, all their separators are different and separate any 
vertex in $A\sm B$ from any vertex in $D\sm C$ minimally.
		This contradicts Lemma~\ref{lem_minimally} and shows (\ref{property_Finiteness}).
	\end{proof}
	
	\begin{lemma}\label{lem_adjacent}
		Let $G$ be a graph and let $(A,B)$, $(C,D)$ and $(E,F)$ be three separations such that 
$(A,B)\nparallel(C,D)$.
		Then the following statements hold:
		\begin{enumerate}[{\rm (i)}]
			\item\label{itm_adjacent1}{\rm \cite[Lemma 2.2]{confing}} If $(E,F)$ is 
nested with $(A,B)$ and with 
			$(C,D)$, then every corner of $(A,B)$ and $(C,D)$ is nested with $(E,F)$;
			\item\label{itm_adjacent2} if $(E,F)$ is nested with 
			$(A,B)$, then there are two adjacent corners of $(A,B)$ and $(C,D)$ which 
are nested with $(E,F)$.
		\end{enumerate} 
	\end{lemma}
	
	\begin{proof}
		For the proof of (\ref{itm_adjacent1}) we refer readers to \cite[Lemma 
2.2]{confing}.
		
		To prove (\ref{itm_adjacent2}), let us assume that $(E,F)$ is nested with $(A,B)$.
		We may assume $(E,F)\leq (A,B)$.
		In particular, we have $E\subseteq A\cup C$ and $E\subseteq A\cup D$ as well as 
$B\cap C\subseteq F$ and $B\cap D\subseteq F$.
		Hence, $(E,F)$ is  nested with $(A\cap C, B\cup D)$ and $(A\cap D, B\cup C)$.
	\end{proof}
	
	Let $(A,B)$ be a separation and let $\Scal$ be a set of separations of a graph~$G$.
	We set
	\[
	\Ccal_\Scal(A,B):= \{(C,D)\in\Scal\mid (A,B)\nparallel(C,D)\},
	\]
	i.\,e.\ $\Ccal_\Scal(A,B)$ is the set of all separations in~$\Scal$ that cross $(A,B)$.
	We set $c_\Scal(A,B):=|\Ccal_\Scal(A,B)|$.
	If $c_\Scal(A,B)$ is finite, then we say that $(A,B)$ has \emph{finite crossing number} 
(with 
	respect to~$\Scal$). 
	Otherwise we say that the \emph{crossing number} of $(A,B)$ is \emph{infinite}.
	
	The following lemma shows that corner separations of crossing separations are crossing with 
less separations than the crossing ones.
	This is an adaption of a result of Dunwoody and Kr\"on \cite[Lemma 5.1]{VertexCuts} to our 
situation.
	
	\begin{lemma}\label{lem_inequality}
		Let $(A,B)$ and $(C,D)$ be two crossing separations of a graph $G$ and let $\Scal$ 
be a set of 
		separations.
		If $(X_1,Y_1)$ and $(X_2,Y_2)$ are two opposite corner separations of $(A,B)$ and 
$(C,D)$, then the 
		following holds:
		\begin{enumerate}[{\rm(i)}]
			\item\label{itm_inequality1} $\Ccal_\Scal(X_1,Y_1) \cap 
\Ccal_\Scal(X_2,Y_2)\subseteq 
			\Ccal_\Scal(A,B)\cap \Ccal_\Scal(C,D)$;
			\item\label{itm_inequality2} $\Ccal_\Scal(X_1,Y_1) \cup 
\Ccal_\Scal(X_2,Y_2)\subsetneq 
			\Ccal_\Scal(A,B)\cup \Ccal_\Scal(C,D)$.
		\end{enumerate}
		In particular, if all sets are finite, then we have
		\[
		c_\Scal(X_1,Y_1) + c_\Scal(X_2,Y_2)< c_\Scal(A,B)+ c_\Scal(C,D).
		\]
	\end{lemma}
	
	\begin{proof}
		For (\ref{itm_inequality1}), let $(E,F)$ be a separation that is nested with either 
$(A,B)$ or 
		$(C,D)$.
		By Lemma \ref{lem_adjacent}\,(\ref{itm_adjacent2}), there are two adjacent corners 
of $(A,B)$ and $(C,D)$ which are nested 
		with $(E,F)$.
		Thus, $(E,F)$ is nested with either $(X_1,Y_1)$ or $(X_2,Y_2)$.
		
		To prove (\ref{itm_inequality2}), let $(E,F)\in \Ccal_\Scal(X_1,Y_1) \cup 
\Ccal_\Scal(X_2,Y_2)$.
		By Lemma~\ref{lem_adjacent}\,(\ref{itm_adjacent1}), the separation $(E,F)$ belongs 
to 
		$\Ccal_\Scal(A,B)\cup \Ccal_\Scal(C,D)$.
		The inclusion is strict, as $(A,B)$ belongs to $\Ccal_\Scal(A,B)\cup 
\Ccal_\Scal(C,D)$ but not to 
		$\Ccal_\Scal(X_1,Y_1) \cup \Ccal_\Scal(X_2,Y_2)$.
		
		The additional assertion follows directly from~(\ref{itm_inequality1}) 
and~(\ref{itm_inequality2}).
	\end{proof}	
	
	\begin{lemma}\label{lem_opposite}
		Let $G$ be a graph and $\Pcal$ be a set of distinguishable $(k+1)$-profiles in~$G$ 
with $k=\kappa(\Pcal,G)$.
		Let $(A_1,A_2), (B_1,B_2)\in\Rcal(k,\Pcal)$.
		Then there are two opposite corner separations of $(A_1,A_2)$ and $(B_1,B_2)$ that 
lie in $\Rcal(k,\Pcal)$.
	\end{lemma}

	\begin{proof}
		For $i\in\{1,2\}$, let $P_{A_i}$ and $P_{B_i}$ be profiles in~$\Pcal$ such that 
$(A_{3-i},A_i)\in P_{A_i}$ and $(B_{3-i},B_i)\in P_{B_i}$.
		Let $\Pcal':=\{P_{A_1},P_{A_2},P_{B_1},P_{B_2}\}$.
		First, we will show the following.
		\begin{equation}\tag{$\dagger$}
			\begin{minipage}[t]{0.9\textwidth}
				\em
				There are $i,j\in\{1,2\}$ and $P,P'\in\Pcal'$ such that 
$(A_{3-i},A_i),(B_{3-j},B_j)\in P$ and $(A_i,A_{3-i}),(B_j,B_{3-j})\in P'$.
			\end{minipage}
		\end{equation}
		If $(B_1,B_2)$ distinguishes $P_{A_1}$ and $P_{A_2}$, then ($\dagger$) holds for 
those two profiles.
		Let us assume that $(B_1,B_2)$ does not distinguish $P_{A_1}$ and $P_{A_2}$.
		We may assume that $(B_1,B_2)\in P_{A_1},P_{A_2}$.
		Without loss of generality, we may assume that $(A_1,A_2)\in P_{B_1}$.
		Then ($\dagger$) holds for the profiles $P_{A_1}$ and $P_{B_1}$.
		
		Let us consider the corner separations $(X_1,X_2):=(A_{3-i}\cup B_{3-j},A_i\cap 
B_j)$ and $(Y_1,Y_2):=(A_i\cup B_j,A_{3-i}\cap B_{3-j})$.
		By definition of a profile, we have $(X_2,X_1)\notin P$ and $(Y_2,Y_1)\notin P'$.
		It is easy to see that the orders of $(X_1,X_2)$ and $(Y_1,Y_2)$ sum to the sum of 
the orders of $(A_1,A_2)$ and $(B_1,B_2)$, which is~$2k$.
		So if one of those corner separations has order at most~$k$, it follows that it 
distinguishes $P$ and~$P'$, and if its order is less than~$k$, this would contradict 
$k=\kappa(\Pcal,G)$.
		Thus, both corner separations have order at least~$k$ and hence exactly~$k$.
		So they lie in $\Rcal(k,\Pcal)$.
	\end{proof}
	
	\begin{lemma}\label{lem_compsep}
		Let $G$ be a well-separable graph with respect to a set $\Pcal$ of principal $k$-profiles, let 
$P,P'\in \Pcal$ and 
		let $(A,B)\in \Reff(P,P')$.
		Then there is a component $X$ of $A\sm B$ such that $(X\cup N(X),V(G)\sm X)$ lies 
in~$\Reff(P,P')$.
	\end{lemma}
	
	\begin{proof}
		Let us assume that $(A,B)\in P$ and $(B,A)\in P'$.
		Let $\{C_i\}_{i\in I}$ be the components of $A\sm B$.
		Since $(C_i\cup N(C_i),V(G)\sm C_i)\leq (A,B)$ and $(A,B)\in P$, we conclude that 
$(C_i\cup 
		N(C_i),V(G)\sm C_i)\in P$ for all $i\in I$.
Since $(D\cup N(D),V(G)\sm D)\leq (B,A)$ for all components $D$ of $B\sm A$, it follows by (P1) and as the profiles are principal that $(V(G)\sm C_i,C_i\cup N(C_i))\in P'$ for some $i\in I$.
		Thus we have 
$(C_i\cup 
		N(C_i),V(G)\sm C_i)\in\Reff(P,P')$.
	\end{proof}
	
	Let $G$ be a graph and $\Pcal$ be a set of profiles of~$G$ each of which is an 
$\ell$-profile for some $\ell>k$.
	Note that every \lc\ separation in~$\Rcal(k,\Pcal)$ has a finite crossing number in the 
subset $\Rcal_{\rm lc}(k,\Pcal)$ of \lc\ separations in~$\Rcal(k,\Pcal)$ by 
Lemma~\ref{lem_minimally}.
	Let $\Ocal(k,\Pcal)$ be the set of \lc\ separations in~$\Rcal(k,\Pcal)$ with minimum 
crossing number in $\Rcal_{\rm lc}(k,\Pcal)$.
	
	The following lemma is essentially already proved in Dunwoody and Kr\"on~\cite{VertexCuts}.
	But their result requires the existence of \emph{cut systems} as they call it.
	Instead of showing that we can define a cut system in our case, we briefly prove 
Lemma~\ref{lem_separation_exists} directly. To state Lemma~\ref{lem_separation_exists}, we define canonicity slightly stronger than in the introduction.

A set $\Scal$ of separations of a graph $G$ is \emph{canonical (with respect to a set $\Pcal$ of profiles)} if the construction $\psi$ of~$\Scal$ commutes with all isomorphisms~$\varphi$ of the pair $(G,\Pcal)$, 
	i.\,e.\ if $\varphi$ is an isomorphism $G\to G'$ that induces a map $\Pcal\to\Pcal'$, then $\varphi'(\psi(G))=\psi(\varphi(G))$, where $\varphi'$ is the map from $\Scal$ to the set of separations of~$G'$ that is induced by~$\varphi$.
	
	\begin{lemma}\label{lem_separation_exists}
		Let $G$ be a well-separable graph with respect to a set $\Pcal$ of robust principal 
$(k+1)$-profiles of~$G$ with $k=\kappa(\Pcal,G)$ such that $\Rcal(k,\Pcal)\neq\es$.
		Then the set $\Ocal(\Pcal,G)$ is nested, canonical with respect to~$\Pcal$ and not 
empty.
	\end{lemma}
	
	\begin{proof}
		Let $(A,B)\in\Ocal(\Pcal,G)$.
		Let us suppose that $(A,B)$ is not nested with an another \lc\ separation $(C,D)$ 
in 
$\Rcal(k,\Pcal)$.
		Then there are two opposite corners whose separations lie in $\Rcal(k,\Pcal)$ by 
Lemma~\ref{lem_opposite} and one of these corner separations, call it $(E,F)$, is crossing with a 
smaller number of separations in $\Rcal(k,\Pcal)$ than $(A,B)$ by Lemma~\ref{lem_inequality}.
		If $(E,F)$ is not \lc, then there is a component $K$ of~$E\sm F$ such that 
$(E',F'):=(K\cup N(K),V(G)\sm K)$ is a \lc\ separation in $\Rcal(k,\Pcal)$ by 
Lemma~\ref{lem_compsep}.
		It is easy to see that $(E',F')$ is nested with all separations that are nested 
with 
$(E,F)$ and thus is crossing with less separations in $\Rcal_{\rm lc}(k,\Pcal)$ than $(A,B)$, a 
contradiction to the choice of $(A,B)$.
	\end{proof}

	\section{Tree-decompositions}\label{sec_TD}
	
	Carmesin, Diestel, Hundertmark and Stein~\cite{confing} presented a method for finite 
graphs 
how to build a \td\ from a 
	nested set of separations.
	Essentially, their method carries over to infinite graphs but the proofs need a small 
adjustment to 
	deal with an additional assumption that we need.
	In this section, we will recap their definitions and results.
	We will omit the proof where appropriate and highlight the differences from the finite to 
the 
	infinite case.
	
	A \emph{\td} of a graph $G$ is a pair $(T, \Vcal)$ of a tree $T$ and a family $\Vcal = 
(V_t)_{t\in 
		V(T)}$ of vertex sets $V_t\subseteq V(G)$, one for every node of $T$, such that:
	\begin{enumerate}[(T1)]
		\item  $V (G) = \bigcup_{t\in V(T)} V_t$,
		\item for every edge $e\in E(G)$, there exists a $t\in V(T)$ with $e\sub V_t$,
		\item $V_{t_1} \cap V_{t_3} \subseteq V_{t_2}$ for all $t_2$ that lie on the 
$t_1$-$t_3$ 
		path in~$T$.
	\end{enumerate}
	The elements of~$\Vcal$ are the \emph{parts} of the \td.
	The sets $V_s\cap V_t$ with $st\in E(T)$ are the \emph{adhesion sets}.
	The \emph{adhesion} is the supremum of the cardinalities of the adhesion sets.
	A \td\ is \emph{$k$-balanced} if all its adhesion sets have size~$k$ and it is 
\emph{balanced} if it is $k$-balanced for some $k\in\N$.
	
	Let $\Ncal$ be a nested set of separations.
	We construct a tree decomposition $(T(\Ncal),\Vcal(\Ncal))$ via $\Ncal$.
	We define a relation $\sim$ on~$\Ncal$:
	\[
	(A,B)\sim (C,D)\ \Leftrightarrow\alt{(A,B) = (C,D) \text{ or}}{(B,A)\text{ is a predecessor 
of 
		}(C,D)\text{ in }(\Ncal,\leq),}
	\]
	where a \emph{predecessor} of an element $z$ of a partial order $(P,\le)$ is an element 
$x<z$ of~$P$ 
	such that there is no $y\in P$ with $x < y < z$.
	By the same argument as in \cite[Lemma 3.1]{confing}, one can show that $\sim$ is an 
equivalence 
	relation on~$\Ncal$.
	The nodes of $T(\Ncal)$ are the equivalence classes of~$\sim$ on~$\Ncal$.
	We define the set of edges of $T(\Ncal)$ as
	\[
	E(T(\Ncal)) := \big\{\{(A,B),(B,A)\}\mid(A,B)\in \Ncal\big\}
	\]
	and an edge is incident with the two equivalence classes of its elements.
	Let $\Vcal(\Ncal)$ consists of all
	\[
	V_t := \bigcap\left\{A\mid (A,B) \in t\right\}
	\]
	with $t\in V(T(\Ncal))$.
	
	\begin{prop}\label{prop_tree}
		Let $G$ be a graph and $\Ncal$ a nested set of separations of~$G$ satisfying {\rm 
			(\ref{property_Finiteness})}.
		Then $T(\Ncal)$ is a tree.
	\end{prop}
	
	\begin{proof}
		The proof of connectedness and lack of having cycles follows the proof of the 
analogue result for 
		finite graphs of Carmesin, Diestel, Hundertmark and Stein~\cite[Theorem 
3.4]{confing} almost verbatim. We just have to apply 
		(\ref{property_Finiteness}) at the according place to verify that $T(\Ncal)$ is 
connected.
	\end{proof}
	
	If $(T,\Vcal)$ is a \td\ of a graph~$G$, $e\in E(T)$ and $T_1,T_2$ the components of $T-e$, 
then 
	$(\bigcup_{t\in V(T_1)}V_t,\bigcup_{t\in V(T_2)}V_t)$ is a separation of~$G$ and its order 
is the 
	size of the adhesion set corresponding to~$e$.
	The separation is the separation \emph{induced by} the edge~$e$.
	
	With Proposition~\ref{prop_tree}, the proof of \cite[Theorem 4.8]{confing} carries over to 
our 
	situation.
	
	\begin{thm}{\rm \cite[Theorem 4.8]{confing}}\label{thm_nestedToTD}
		Let $G$ be a graph and $\Ncal$ a nested set of separations satisfying~{\rm 
			(\ref{property_Finiteness})}.
		Then $(T(\Ncal),\Vcal(\Ncal))$ is a \td.
		The separations induced by $(T(\Ncal),\Vcal(\Ncal))$ are precisely those 
in~$\Ncal$.\qed
	\end{thm}

	We say that a profile $P$ \emph{lives in} a part $V_t$ of a \td\ $(T,\Vcal)$ if for every 
separation 
	$(A,B)$ that is induced by~$(T,\Vcal)$ we have $(A,B)\in P$ if $V_t\sub B$.
	Note that consistency of~$P$ implies that if $P$ lives in a part $V_t$, then $V_t$ is no 
adhesion 
	set.
		
	\section{Profiles and parts of \td s}\label{sec_Parts}
	
	Let $(T,\Vcal)$ be a \td\ of a graph~$G$.
	For $t\in V(T)$, the \emph{torso} $H$ of~$V_t$ is the subgraph of~$G$ induced by~$V_t$ with 
additional edges $xy$ for all $x,z\in V_t$ that lie in a common adhesion set in~$V_t$.
	A separation $(A,B)$ of~$G$ induces a separation
	\[
	(A_H,B_H):=(A\cap V_t,B\cap V_t)
	\]
	of~$H$ if and only if $(A,B)$ does not separate any adhesion set in~$H$.
	Note that a proper separations of~$G$ need not induce proper separations of~$H$.
	For a set $\s$ of separations of~$G$ let
	\[
	\s_H:=\{(A_H,B_H)\mid (A,B)\in \s, (A_H,B_H)\text{ is a separation of }H\}
	\]	
	be the set of separations induced by~$\s$ on~$H$.
	That way, a profile of~$G$ induces a set of separations of~$H$.
	In the following proposition, we will see that this induced set of separations is indeed a 
profile in 
	the cases we are interested in.
	Furthermore, we will prove that the induced separations on a torso separate the induced 
profiles in 
	a best possible way.
	
	\begin{prop}\label{prop_ProfilesInTD}
		Let $G$ be a graph, let $k\in\N$ and let $\Pcal$ be a set of robust principal profiles of~$G$ 
each of which is 
		an $\ell$-profile for some $\ell> k$.
		Let $(T,\Vcal)$ be a \td\ of adhesion at most $\kappa(\Pcal,G)$ such that $N(C)=S$ 
for every 
		adhesion set $S$ of $(T,\Vcal)$ and every component $C$ of $G-S$.
		Let $V_t$ be a part of~$(T,\Vcal)$ and $H$ be its torso.	Assume that all separations induced by edges of the decomposition tree are proper.
		Then the following hold.
		\begin{enumerate}[{\rm (i)}]
			\item\label{itm_ProfilesInTD1} For every (robust) principal $\ell$-profile $P$ that lives 
in~$V_t$, where $\ell$ 
			is larger than the adhesion of $(T,\Vcal)$, the set $P_H$ is a (robust) principal 
$\ell$-profile of~$H$;
			\item\label{itm_ProfilesInTD2} 
					for all separations $(A,B)$ of~$G$ with 
$A\cap B\sub V_t$ that distinguish profiles 
			in~$\Pcal$ that live in~$V_t$, the pair $(A\cap V_t,B\cap V_t)$ is a 
			separation of~$H$ of the same order and distinguishes the induced profiles;
			\item\label{itm_ProfilesInTD3} for all distinguishable profiles $P_1,P_2\in 
\Pcal_H$ and all distinguishable profiles $Q_1,Q_2\in\Pcal$ such that 
			$Q_i$ induces $P_i$ for $i=1,2$, there is a separation $(A,B)$ of~$G$ that 
distinguishes $Q_1$ and 
			$Q_2$ efficiently such that $({A\cap V_t},{B\cap V_t})$ is a separation of 
order $|A\cap B|$ that 
			distinguishes $P_1,P_2$ efficiently;
			\item\label{itm_ProfilesInTD4} if the subset $\Pcal'$
			of~$\Pcal$ consisting of all profiles that live in~$V_t$ contains at least two elements, then 
$\kappa(\Pcal_H,H)=\kappa(\Pcal',G)$.
		\end{enumerate}
	\end{prop}
	
	\begin{proof}
		Let $(A,B)$ be a separation of the graph $G$ whose separator lies in~$V_t$.
		Set $(A_H,B_H):=(A\cap V_t,B\cap V_t)$.
		Since the separator $A\cap B$ is included in the vertex $V(H)$ of the torso $H$, we have $|A_H\cap B_H|=|A\cap B|$.
		
		Suppose for a contradiction $(A_H,B_H)$ is not a separation of the torso~$H$. 
		Then there is an edge $ab$ in~$H$ with 
$a\in A\sm B$ and 
		$b\in B\sm A$. This edge does not lie in~$G$ and hence both its end vertices lie in a common 
adhesion set. Let $x$ be an edge of the decomposition tree whose adhesion set contains the vertices $a$ and $b$. Let $(X,Y)$ be the separation corresponding to the edge $x$. Either the vertex set $X\sm Y$ or $Y\sm X$ is disjoint from the part $V_t$. By symmetry, we may assume and we do assume that $X\sm Y$ is disjoint from $V_t$. As $(X,Y)$ is proper by assumption, the set $X\sm Y$ is nonempty. Let $C$ be a component of $G-(X\cap Y)$ included in $X\sm Y$.  By the assumptions on $(T,\Vcal)$ and its adhesion sets, we have $X\cap Y=N(C)$.
So $C$ includes a path whose endvertices are adjacent to the vertices $a$ and $b$. As this path is disjoint from the separator $A\cap B$, this violates our assumption that $(A,B)$ is a separation of the graph $G$. Hence it must be that $(A_H,B_H)$ is a separation of the torso~$H$.

		Let $(A,B)$ be a separation of~$H$.
		By definition, all adhesion sets of $(T,\Vcal)$ induce complete graphs in~$H$.
		Thus, every adhesion set lies either in~$A$ or in~$B$.
		We define a separation of~$G$ by adding all components $C$ of $G-V_t$ to~$A$ if 
$N(C)\sub A$ and to 
		$B$ otherwise.
		It is easy to see that the resulting pair $(A^G,B^G)$ is a separation of~$G$.
		Its order is $|A\cap B|$ since we did not add anything to $A$ and $B$ 
simultaneously.
		So if $(A^G,B^G)$ lies in a profile $P\in\Pcal$, then $(A,B)\in P_H$.
		
		Let $P$ be an $\ell$-profile that lives in~$V_t$ such that $\ell$ is larger than 
the 
adhesion of 
		$(T,\Vcal)$.
		First, we will show that $P_H$ is consistent.
		Let $(A_H,B_H)$ and $(C_H,D_H)$ be separations of~$H$ such that 
$(C_H,D_H)\leq(A_H,B_H)$ and 
		$(A_H,B_H)\in P_H$.
		Suppose that $(D_H,C_H)$ lies in~$P_H$.
		So we may assume that there are $(A,B),(D,C)\in P$ whose induced separations in~$H$ 
are 
		$(A_H,B_H),(D_H,C_H)$, respectively.
		Note that we have $A\cap B=A_H\cap B_H$ and $C\cap D=C_H\cap D_H$ by definition of 
induced separations.
Since $P$ is a principal $\ell$-profile, there is a component $K$ of $C\sm D$ with $(V(G)\sm K,K\cup N(K))\in P$.
If $K\cap V_t\neq\es$, then there is a vertex in $K\cap C_H\sub K\cap (A_H\sm B_H)$.
Since $A\cap B=A_H\cap B_H\sub D_H$, we conclude $K\sub A$.
As $(K\cup N(K),V(G)\sm K)\leq (A,B)\in P$, the separation $(K\cup N(K),V(G)\sm K)$ lies in~$P$ by~(P1).
This is a contradiction to (P1) for~$P$.
So $K$ contains no vertex of~$C_H$.
In particular $K\cap V_t=\es$.
Let $e$ be an edge of~$T$ that is incident with~$t$ and such that for the separation $(A_e,B_e)$ induced by~$e$ we have $V_t\sub A_e$ and $K\sub B_e$.
Then we have $(A_e,B_e)\leq (V(G)\sm K,K\cup N(K))\in P$.
So (P1) implies $(A_e,B_e)\in P$ and hence $P$ does not live in~$t$, a contradiction to its choice.
		This shows that $P_H$ is consistent.
		
		Let $(A_H,B_H),(C_H,D_H)\in P_H$ and let $(A,B),(C,D)\in P$ induce them.
		By (P2), we have $(E,F):=(B\cap D,A\cup C)\notin P$.
If the order of $(E,F)$ is less than~$k$, then $(F,E)\in P$ and hence $(F_H,E_H)\in P_H$.
So consistency of~$P_H$ implies $(E_H,F_H)\notin P_H$.
If the order of $(E,F)$ is at least~$k$, then the same holds for the order of $(E_H,F_H)$ and hence it does not lie in~$P_H$.
		Thus, (P2) follows for~$P_H$ and $P_H$ is a profile.
		Following the construction of $(A^G,B^G)$ for a separation $(A,B)$ of~$H$ shows 
that 
$P_H$ is an 
		$\ell$-profile.

As $P$ is a principal $\ell$-profile, there is for every vertex set $S$ of size less than~$k$ a component $K$ of $G-S$ such that $(A,B):=(V(G)\sm K,K\cup N(K))\in P$.
As in the proof of (P1) for~$P_H$, it follows that $(A_H,B_H)$ is a proper separation of~$H$.
Then each separation $(C,D)$ with $S=(C,D)$ lies in~$P_H$ if and only if $D$ contains~$B_H$.
It follows that $P_H$ is principal.
		
		Now let $P$ be $k$-robust.
		Let $(A_H,B_H)\in P_H$ and let $(C_H,D_H)$ be a separation of~$H$ of order at 
most~$k$.
		Then there are $(A,B)\in P$ and a separation $(C,D)$ of order at most~$k$ that 
induce $(A_H,B_H)$ and $(C_H,D_H)$ in~$H$, 
		respectively.
		Suppose that $(B_H\cap C_H,A_H\cup D_H)$ and $(B_H\cap D_H,A_H\cup C_H)$ are 
in~$P_H$ and both have 
		order less than $|A\cap B|$.
		As $P$ is $k$-robust, we may assume that $(B\cap C,A\cup D)\notin P$.
		As the orders of $(B\cap C,A\cup D)$ and $(B_H\cap C_H,A_H\cup D_H)$ coincide, we 
conclude $(A\cup 
		D, B\cap C)\in P$.
		So by definition, $(A_H\cup D_H,B_H\cap C_H)\in P_H$ which contradicts consistency 
of~$P_H$ as 
		$(B_H\cap C_H,A_H\cup D_H)\in P_H$.
		Thus, (\ref{itm_ProfilesInTD1}) holds.
		
		We have already seen that separations $(A,B)$ whose separators lie in~$V_t$ induce 
separations 
		in~$H$ and it is obvious from the definitions that, if $(A,B)$ distinguishes two 
profiles, the 
		induced separation distinguishes the induced profiles.
		Thus, (\ref{itm_ProfilesInTD2}) holds.
		
		To prove~(\ref{itm_ProfilesInTD3}), let $P_1,P_2\in \Pcal_H$ and $Q_1,Q_2\in\Pcal$ 
such that $Q_i$ 
		induces $P_i$ for $i=1,2$.
		Let $(A,B)$ be a separation of~$G$ such that $(A_H,B_H)$ distinguishes $P_1$ and 
$P_2$ efficiently.
		Since any separation $(C,D)$ of~$G$ that separates $Q_1$ and $Q_2$ efficiently, 
also 
induces a 
		separation $(C_H,D_H)$ that distinguishes $P_1$ and~$P_2$, we conclude $|A\cap 
B|\leq|C\cap D|$ and 
		thus we may assume $(A,B)\in Q_1$ and either $(A,B)$ or $(B,A)$ lies in~$Q_2$.
		If $(A,B)\in Q_2$, then we conclude $(A_H,B_H)\in P_1\cap P_2$.
		But as it distinguishes $P_1$ and $P_2$, we also have $(B_H,A_H)\in P_2$, a 
contradiction to 
		consistency of~$P_2$.
		Thus, $(A,B)$ distinguishes $Q_1$ and $Q_2$.
		Any separation of smaller order than $(A,B)$ that distinguishes $Q_1$ and $Q_2$ 
induces a 
		separation of smaller order than $(A_H,B_H)$ distinguishing $P_1$ and~$P_2$.
		Thus, $(A,B)$ distinguishes $Q_1$ and $Q_2$ efficiently.
		This shows~(\ref{itm_ProfilesInTD3}).
		
		Finally, (\ref{itm_ProfilesInTD4}) follows immediately from 
(\ref{itm_ProfilesInTD2}) and 
		(\ref{itm_ProfilesInTD3}).
	\end{proof}
	
	\begin{lemma}\label{lem_torso}
		Let $G$ be a well-separable graph with respect to a set $\Pcal$ of robust principal profiles 
of~$G$ of order $k+1$ with $k=\kappa(\Pcal,G)$.
		Let $(T,\Vcal)$ be a $k$-balanced \td.
		Let $X$ be a torso of~$(T,\Vcal)$ and let $\Pcal'$ be the set of profiles in~$X$ 
that are induced by profiles in~$\Pcal$.
		Then $\Scal(k,\Pcal',X)=\emptyset$ and $\Rcal(k-1,\Pcal',X)=\emptyset$.
	\end{lemma}

	\begin{proof}
		Proposition~\ref{prop_ProfilesInTD} implies $\Rcal(k-1,\Pcal',X)=\es$.
		Let $(A,B)\in\Rcal(k,\Pcal',X)$ and let $C$ be a component of $X-(A\cap B)$.
		By Proposition~\ref{prop_ProfilesInTD}\,(\ref{itm_ProfilesInTD3}), there is a 
separation $(A^G,B^G)$ that induces $(A,B)$ on~$X$ such that it distinguishes those profiles 
of~$\Pcal$ that induce profiles of~$\Pcal'$ that are distinguished by $(A,B)$.
	
		Let us suppose that there is a component $C$ of $X-(A\cap B)$ with $N(C)\subsetneq 
A\cap B$.
		Then this lies in a component $K$ of $G-(A\cap B)$ with $N(K)\subsetneq A\cap B$, 
which is impossible because of $\Scal(k,\Pcal)=\es$.
	\end{proof}

	Now we are going to construct a \td\ of an arbitrary graph with a set of profiles such that 
the \td\ 
	has a unique part in which all profiles live and whose torso is well-separable with respect 
to the 
	induced profiles.
	
	\begin{prop}\label{prop_makeNice}
		Let $G$ be a graph and let $\Pcal$ be a set of robust principal profiles of~$G$ that are pairwise 
		$\kappa(\Pcal,G)$-distinguishable and each of which is an $\ell$-profile for some 
		$\ell>\kappa(\Pcal,G)$.
		Then there exists a \td\ $(T,\Vcal)$ of adhesion less than $\kappa(\Pcal,G)$ that 
is 
canonical with respect to~$\Pcal$ such that 
		there exists a unique part $V_t$ of $(T,\Vcal)$ in which all profiles 
of~$\Pcal$ live and such that 
		the torso $H$ of~$V_t$ is well-separable with respect to~$\Pcal_H$. Moreover, all separations corresponding to edges of the decomposition tree are proper. 
	\end{prop}
	
	\begin{proof}
		Let $\Ccal$ be the set of all degenerated components of separations 
$(A,B)\in\Reffk(\Pcal,G)$.
		Let $T$ be a star with $|\Ccal|$ leaves.
		Let $x$ be the central vertex of~$T$ and let $\varphi$ be a bijection from the set 
of leaves of~$T$ 
		to~$\Ccal$.
		For a leaf $y$, we set $V_y:= \varphi(y)\cup N(\varphi(y))$ and we set $V_x:= 
V(G)\sm \bigcup\Ccal$. 
		We claim that $(T,\Vcal)$ with $\Vcal:=\{V_z\mid z\in V(T)\}$ is a \td\ of adhesion 
less than 
		$\kappa(\Pcal,G)$.
		
		Since the adhesion set of an edge $xy$ has size at most $N(\varphi(y))$, it follows 
from the choice 
		of~$\Ccal$ that the adhesion is less than $\kappa(\Pcal,G)$.
		To show that $(T,\Vcal)$ is a \td, it suffices to show (T3).
		This follows immediately once we showed the following two properties.
		\begin{enumerate}[(a)]
			\item\label{itm_pf_makeNice1} Distinct elements of~$\Ccal$ are disjoint;
			\item\label{itm_pf_makeNice2} elements of~$\Ccal$ are disjoint from 
separators of separations in 
			$\Reffk(\Pcal,G)$.
		\end{enumerate}
		Let $(A,B)$ and $(C,D)$ be two separations in $\Reffk(\Pcal,G)$ and let $X$ be a 
degenerated component of 
		$G-(A\cap B)$.
		If there is also a degenerated component $Y$ of $G-(C\cap D)$ with $X\cap Y\neq\es$ 
that is distinct 
		from~$X$, then either $X$ intersects $C\cap D$ or $Y$ intersects $A\cap B$.
		Thus, (\ref{itm_pf_makeNice2}) implies (\ref{itm_pf_makeNice1}) and it remains to 
prove  
		(\ref{itm_pf_makeNice2}).
		
		Let us suppose $X\cap (C\cap D)\neq\es$ for some $(C,D)\in\Reffk(\Pcal,G)$.
		Without loss of generality, we may assume $X\sub B$.
		Let $P,P',Q,Q'\in\Pcal$ such that $(A,B)$ distinguishes $P$ and $P'$ efficiently 
and 
$(C,D)$ 
		distinguishes $Q$ and~$Q'$ efficiently.
		We may assume $(A,B)\in P$ and $(C,D)\in Q$.
		We will show that one corner separation, either $(B\cap C, A\cup D)$ or $(B\cap D, 
A\cup C)$, has 
		order at most $|A\cap B|$.
		Let us suppose that both have order at least $|A\cap B|+1$.
		As the orders of opposite corner separations sum to $2|A\cap B|$, the orders of 
$(A\cap C,B\cup D)$ 
		and $(A\cap D, B\cup C)$ are less than $|A\cap B|$.
		Since these two corner separations are less than $(A,B)$, they lie in~$P$.
		But neither $(B\cup D,A\cap C)$ nor $(B\cup C, A\cap D)$ can lie in~$P'$ as their 
orders are less 
		than $|A\cap B|$ but $(A,B)\in \Reffk(\Pcal,G)$.
		So $(A\cap C,B\cup D)$ and $(A\cap D,B\cup C)$ lie in~$P'$.
		Then (P2) implies $(B,A)\notin P'$.
		This contradiction shows that either $(B\cap C, A\cup D)$ or $(B\cap D, A\cup C)$ 
has order at most $|A\cap B|$.
		Let us denote this separation by $(E,F)$.
		Note that $X\cap(C\cap D)$ lies in the separator of $(E,F)$.
		By Lemma~\ref{lem_compsep} there is a component $K$ of $E\sm F$ such that $(K\cup 
N(K),V(G)\sm K)$ 
		lies in $\Reffk(\Pcal,G)$.
		Then we have $X\cap(C\cap D)\sub N(K)$.
		This implies $K\sub X$.
		But then $(X\cup N(X),V(G)\sm X)$ distinguishes two profiles in~$\Pcal$ while its 
order is less than 
		$\kappa(\Pcal,G)$, which is a contradiction.
		This shows~(\ref{itm_pf_makeNice2}).
		So we have verified that $(T,\Vcal)$ is a \td.
		It is canonical as we never made a choice in its construction.
		
		Since the adhesion of $(T,\Vcal)$ is less than $\kappa(G,\Pcal)$, it only remains 
to 
show that the 
		torso $H$ of~$V_x$ is well-separable with respect to~$\Pcal_H$.
		
		Proposition~\ref{prop_ProfilesInTD}\,(\ref{itm_ProfilesInTD4}) implies 
		$\kappa(\Pcal,G)=\kappa(\Pcal_H,H)$.
		Let $(A,B)\in\Reffk(\Pcal_H,H)$.
		Let $(A^G,B^G)$ be a separation of~$G$ of order $|A\cap B|$ that induces $(A,B)$ 
on~$H$ as 
		constructed in the proof of Proposition~\ref{prop_ProfilesInTD}: if $N(C)\sub A$ 
for 
a component $C$ 
		of $G\setminus V_x$, add $C$ to~$A$ and otherwise to~$B$.
		Then $(A^G,B^G)\in\Reffk(\Pcal,G)$.
		Let us suppose that there is a component $C$ of $H-(A\cap B)$ with $N(C)\subsetneq 
A\cap B$.
		Let $C^G$ be the component of~$G-(A\cap B)$ that contains~$C$.
		Since $C^G\cap V_x\neq\es$, the construction of $(T,\Vcal)$ implies $N(C^G)=A\cap 
B$.
		Let $u\in N(C^G)\sm N(C)$ and $v\in C^G$ a neighbour of~$u$.
		Let $u'\in N(C)$ and $v'$ be a neighbour of~$u'$ in~$C$.
		Let $P$ be a $u$-$u'$ path all of whose inner vertices lie in~$C^G$ and whose first 
and last edges 
		are $uv$, $v'u'$, respectively.
		We construct a $u$-$u'$ path $P'$ in~$H$: whenever $P$ leaves $V_x$, it does so 
through an adhesion 
		set and must reenter $V_x$ through the same adhesion set by (T3).
		We then replace this subpath by the edge between its two end vertices, which lies 
in~$H$.
		The resulting path has no vertices outside of~$P$ and thus no vertex of $A\cap B$.
		But it contains all vertices of~$P$ that lie in~$V_x$.
		So it contains~$v'$ and, thus, it contains a vertex of~$C$.
		Hence, all inner vertices of~$P'$ lies in~$C$.
		This contradicts $u\notin N(C)$ and shows that $H$ is well-separable with respect 
to~$\Pcal_H$.

The \lq Moreover-part\rq, directly follows from (a) and (b). 
	\end{proof}
	
	\section{The case: fixed $k$}\label{sec_k}
	
	Let $\Ncal$ be a set of separations and $\Pcal$ be a set of profiles of a graph~$G$.
	We call $\Ncal$ \emph{nice} (for $\Pcal$) if it is a nested set of \lc\ separations of 
order 
$k$ in $R(k,\Pcal)$ and we call $\Ncal$ \emph{distinguishing} if it distinguishes all 
$k$-distinguishable pairs of profiles in~$\Pcal$.
	An \emph{$\Ncal$-block} $X$ is a maximal subset of $V(G)$ such that for every 
$(A,B)\in\Ncal$ we have either $X\sub A$ or $X\sub B$ but not both.
	Note that $X$ is the intersection of all sides $A$ for $(A,B)\in\Ncal$ that contain~$X$.
	For an $\Ncal$-block~$X$, its \emph{torso} is the graph induced by~$X$ in~$G$ with 
additional edges $xy$ whenever $x$ and~$y$ lie in a separator $A\cap B$ of a separation 
$(A,B)\in\Ncal$ with $A\cap B\sub X$.
	
	We call $\Ncal$ \emph{extendable (for $\Pcal$)} if for any two (distinct)
	robust profiles in~$\Pcal$ of the same order, there is some separation of~$G$ 
distinguishing 
these two profiles efficiently that is nested with~$\Ncal$.
	
	\begin{thm}\label{thm_canon_td}
		Let $G$ be a well-separable graph with respect to a set $\Pcal$ of robust principal 
$(k+1)$-profiles of $G$ with $k=\kappa(\Pcal,G)$.
		Then $G$ has a \td\ $(T,\Vcal)$ satisfying the following.
		\begin{enumerate}
			\item $(T,\Vcal)$ distinguishes any two $k$-distinguishable robust profiles 
in $\Pcal$;
			\item $(T,\Vcal)$ is canonical with respect to~$\Pcal$;
			\item $(T,\Vcal)$ is $k$-balanced.
		\end{enumerate}
	\end{thm}
	
	\begin{proof}
		Our first aim is to construct a canonical set $\Ncal$ that is nice and 
distinguishing.\footnote{The existence of a non-canonical set $\Ncal$ follows from {\cite[Theorem 
5.9]{{topo_ends}}}. Here we show how the proof of that theorem can be modified to give a canonical 
set $\Ncal$.}		
		We construct the set $\Ncal$ by transfinite recursion.
		We set $\Ncal_0:=\emptyset$. 
		Assume we already constructed all sets $\Ncal_\alpha$ for $\alpha<\beta$ such that 
they are nice and canonical.
		If $\beta$ is a limit ordinal we set $\Ncal_\beta:=\bigcup_{\alpha<\beta} 
\Ncal_\alpha$.
		This set is nice and canonical as so are all sets $\Ncal_\alpha$.
		
		Now assume that $\beta=\gamma+1$ is a successor ordinal.
		If $\Ncal_\gamma$ is distinguishing, we stop and set $\Ncal:=\Ncal_\gamma$. 
Otherwise there are robust profiles $P,Q\in \Pcal$ that are distinguished by a separation of order 
$k$ in~$G$ but not by $\Ncal_\gamma$.
		Hence the profiles $P$ and $Q$ have 
		to live in the same $\Ncal_\gamma$-block~$X$.
		Let $X'$ be the torso of~$X$.
		By {\cite[Theorem 5.9]{{topo_ends}}} the set $\Ncal_\gamma$ is extendable 
for~$\Pcal$.
		Let $\Pcal'$ be the set of profiles induced by $\Pcal$ on~$X'$.
		By Proposition~\ref{prop_ProfilesInTD}\,(\ref{itm_ProfilesInTD1}), the robust 
profiles $P$ and $Q$ induce robust profiles in the set 
		$\Pcal'$; in particular, $\Pcal'$ is not empty.
		By Proposition~\ref{prop_ProfilesInTD}\,(\ref{itm_ProfilesInTD3}), the set 
$R(k,\Pcal',X')$ of relevant separations is not empty.
		By Lemma~\ref{lem_torso} $S(k,\Pcal')=\emptyset$ and $R(k-1,\Pcal')=\emptyset$.
		So by Lemma \ref{lem_separation_exists} the set $\Mcal(X)$ of \lc\ separations in 
$\Rcal(k,\Pcal',X')$ with minimum crossing number in $\Rcal_{\rm lc}(k,\Pcal',X')$ is nested and 
not 
empty.

		Similarly, as in the proof of Proposition~\ref{prop_ProfilesInTD}, we extend the 
set 
$\Mcal(X)$ of separations of the torso $X'$ to a set of separations in the graph $G$:
		given $(A,B)\in \Mcal(X)$, we obtain $A^G$ from~$A$ by adding all components of $G- 
X$ that have a neighbour in $A\sm B$ and we obtain $B^G$ from $B$ by adding all other components.
		By construction $(A^G,B^G)$ is a \lc\ separation of~$G$ of order~$k$.
		Let $\Mcal^G(X)$ be the set of all separations $(A^G,B^G)$. 
		We set $\Ncal_\beta:=\Ncal_\gamma\cup \bigcup_Y\Mcal^G(Y)$, where the union ranges 
over all $\Ncal_\gamma$-blocks $Y$ such that there are at least two profiles of~$\Pcal$ living 
in~$Y$.
		This definition ensures that $\Ncal_\beta$ is canonical.
		Let us prove			
		\begin{equation}\tag{$\ddagger$}
		\Ncal_\beta \text{ is nice.}
		\end{equation}
		
		The separation $(A^G,B^G)$ is in $\Rcal(k,\Pcal,G)$ as it distinguishes the two 
robust profiles whose induced profiles in~$\Pcal'$ are distinguished by $(A,B)$ in~$X'$.
		It remains to prove that $\Ncal_\beta$ is nested.
		
		Any separation in $\Mcal^G(X)$ is nested with any separation in $\Ncal_\gamma$ by 
{\cite[Observation 4.22]{{topo_ends}}}.
		The same result {\cite[Observation 4.22]{{topo_ends}}} also implies that every 
separation in $\Mcal^G(Y_1)$ is nested with any separation in $\Mcal^G(Y_2)$ if $Y_1\neq Y_2$.

		Thus it suffices to show that any two separations $(A_1^G,B_1^G)$ and 
$(A_2^G,B_2^G)$ in $\Mcal^G(X)$ are nested.
		The separations $(A_1,B_1)$ and $(A_2,B_2)$ are nested as they lie in $\Mcal(X)$.
		
		Let us consider the case $A_1\sub A_2$ and $B_2\sub B_1$ first.
		Let $X$ be a component with a neighbour $x\in A_1\sm B_1$.
		If $x\notin A_2\sm B_2$, then we have $x\in B_2\sub B_1$, which is a contradiction.
		Thus, every component of $G-X$ that has a neighbour in $A_1\sm B_1$ also has a 
neighbour in $A_2\sm B_2$.
		So we have $A_1^G\sub A_2^G$ and $B_2^G\sub B_1^G$.

		The case $A_2\sub A_1$ and $B_1\sub B_2$ is analogous to the previous case.
		
		Let us now consider the case $A_1\sub B_2$ and $A_2\sub B_1$.
		Then for every component $K$ of $G-X$ that has a neighbour in~$A_1\sm B_1$ we have 
$N(K)\sub A_1\sub B_2$ as $(A_1,B_1)$ is a separation of the torso~$X'$; in particular $N(K)\cap 
(A_2\sm B_2)=\es$.
		So $K$ does not lie in $A_2^G$ but in $B_2^G$.
		This shows $A_1^G\sub B_2^G$.
		An analogous argument shows $A_2^G\sub B_1^G$.
		
		The last case to consider is $B_1\sub A_2$ and $B_2\sub A_1$.
		If $(A_1,B_1)=(B_2,A_2)$, then we are in the above case $A_1\sub B_2$ and $A_2\sub 
B_1$. So we may assume that those two separations are distinct; in particular $B_1\subsetneq A_2$ 
and $B_2\subsetneq A_1$.
		First we show that there is a vertex $v$ in $(A_2\cap B_2)\sm B_1$, see 
Figure~\ref{fig:corner1}.
		As the separation ${(A_2,B_2)}$ is proper, there is a vertex $w$ in $B_2\sm A_2$.
		This vertex $w$ must be contained in $A_1\sm B_1$ as $B_1\subseteq A_2$.
		As the connected set $A_1\sm B_1$ is not a subset of $B_2$, there must be a vertex 
$v$ in the link $(A_2\cap B_2)\sm B_1$. 
		
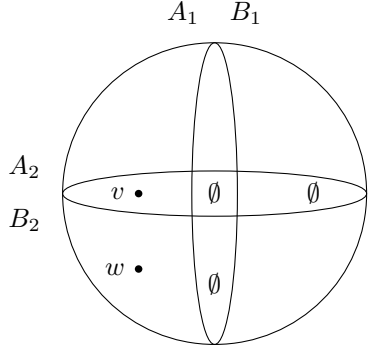
\begin{figure}[htpb]  
\begin{center}
\begin{tikzpicture}
\draw (0,0) ellipse (0.3cm and 2cm);
\draw (0,0) ellipse (2cm and 0.3cm);
\draw (0,0) circle (2cm);
\draw (0,2.4) node [inner sep=2pt,label={right:$B_1$}] {};
\draw (0,2.4) node [inner sep=2pt,label={left:$A_1$}] {};
\draw (-2.5,0) node [inner sep=2pt,label={above:$A_2$}] {};
\draw (-2.5,0) node [inner sep=2pt,label={below:$B_2$}] {};
\draw (-1,0) node [circle,fill, inner sep=1pt,label={left:$v$}] {};
\draw (-1,-1) node [circle,fill, inner sep=1pt,label={left:$w$}] {};
\draw (0,-1.5) node [inner sep=1pt,label={$\emptyset$}] {};
\draw (1.3,-0.3) node [inner sep=1pt,label={$\emptyset$}] {};
\draw (0,-0.3) node [inner sep=1pt,label={$\emptyset$}] {};
\end{tikzpicture}\caption{The corner diagram of ${(A_1,B_1)}$ and ${(A_2,B_2)}$. By 
nestedness the bottom right corner and its two adjacent links are empty. The vertex $w$ exists as 
$(A_2,B_2)$ is proper. The vertex $v$ exists as the connected set $A_1\sm B_1$ is not a subset of 
$B_2$. 
}\label{fig:corner1}
\end{center}
\end{figure}
		
		Our aim is to show $B_2^G\subseteq A_1^G$.
		So let $K$ be a component of $G- X$ that is a subset of $B_2^G$.
		If $K$ has a neighbour in $B_2\sm A_2$, then it also has a neighbour in the 
superset 
$A_1\sm B_1$, and thus $X$ is included in $A_1^G$.
		Thus we may assume that $K$ has no neighbour in $B_2\sm A_2$.
		By its choice $K$ has no neighbour in $A_2\sm B_2$.
		So all its neighbours are in the separator $A_2\cap B_2$.
		As $(A_2,B_2)$ is in $\Rcal(k,\Pcal)$ and $\Scal(k,\Pcal)$ is empty, the component 
$K$ has the whole separator $A_2\cap B_2$ in its neighbourhood. 
		Thus the vertex $v$ lies in the neighbourhood of $K$ but also in $A_1\sm B_1$. 
		So $K$ is included in $A_1^G$.
		We have shown $B_2^G\sub A_1^G$.
		An analogous argument shows $B_1^G\sub A_2^G$.
		Thus $(A_1^G,B_1^G)$ and $(A_2^G,B_2^G)$ are nested.
		
		All cases combined show that $\Ncal_\beta$ is nested, which proves~($\ddagger$).
		
		Every separation $(A^G,B^G)$ in $\Mcal^G(X)$ distinguishes two robust profiles 
in~$\Pcal$ that are not distinguished by~$\Ncal_\gamma$.
		Indeed, there are profiles in~$\Pcal$ inducing profiles in the torso~$X'$ such that 
$(A,B)$ distinguishes these induced profiles. 
		So $\Mcal^G(X)$ contains a separation distinguishing two profiles not distinguished 
by~$\Ncal_\gamma$.
		So $\Ncal_\beta$ is strictly larger than $\Ncal_\gamma$.		
		As the sequence of the $\Ncal_\alpha$ is strictly increasing but the number of 
separations of order $k$ is bounded, this recursions has to stop eventually.
		So there is some canonical, nice and distinguishing set~$\Ncal$. 
		
		By Lemma~\ref{lem_make-tree} the set $\Ncal$ has property 
(\ref{property_Finiteness}) and thus induces by Theorem~\ref{thm_nestedToTD} a canonical \td\ that 
distinguishes all robust profiles in~$\Pcal$.
		It is $k$-balanced by construction. 
	\end{proof}	
	
	\section{Trees of \td s}\label{sec_TTD}
	
	In this section, we will prove the main theorem of this paper.
	Before we do that, we give the formal definition of trees of \td s.
	
	A \emph{rooted} tree is a pair $(T,r)$ of a tree $T$ and a vertex $r\in V(T)$, called the 
	\emph{root} of~$T$.
	The \emph{level} of a vertex $t\in V(T)$ is $d(t,r)+1$.
	A \emph{\ttd} of a graph $G$ is a triple $((T,r),(G_t)_{t\in V(T)},(T_t,\Vcal_t)_{t\in 
V(T)})$ of a 
	rooted tree $(T,r)$, a family $(G_t)_{t\in V(T)}$ of graphs, one for every node of~$T$ with 
$G_r=G$, 
	and a family $(T_t,\Vcal_t)_{t\in V(T)}$ such that for every node $t\in V(T)$ the pair 
	$(T_t,\Vcal_t)$ is a \td\ of~$G_t$ such that for every node $t\in V(T)$ the graphs assigned 
to its 
	neighbours on the next level are distinct torsos of $(T_t,\Vcal_t)$.
	The \ttd\ \emph{distinguishes} two profiles (\emph{efficiently}) if there exists a node 
$t\in V(T)$ such that some separation of~$G_t$ induced by $(T_t,\Vcal_t)$ distinguishes the induced 
profiles (efficiently) and is induced by a separation of~$G$ that distinguishes the profiles 
(efficiently).	

\begin{remark}\emph{
 It was proved first in \cite{CG17} that any tree of \td s 
of a finite graph can be stuck together into a single \td\ (yet with somewhat 
different notation). The proof is essentially the same as that in Proposition \ref{prop_TDrefinement}.}
\end{remark}

	\begin{thm}\label{thm_ttd}
		Let $G$ be a graph and $\Pcal$ a set of distinguishable robust principal profiles each of 
which is an 
		$\ell$-profile for some $\ell\in\N\cup\{\aleph_0\}$.
		Then there exists a \ttd\ $((T,r),(G_t)_{t\in V(T)},(T_t,\Vcal_t)_{t\in V(T)})$ 
that 
is canonical with respect to~$\Pcal$ such that 
		the following hold.
		\begin{enumerate}[{\rm (1)}]
			\item\label{itm_ttd_distinguish} The \ttd\ distinguishes $\Pcal$ 
efficiently;
			\item\label{itm_ttd_evenbalanced} if $t\in V(T)$ has level $2k$, then 
$(T_t,\Vcal_t)$ is 
			$k$-balanced;
			\item\label{itm_ttd_evenneighbours} nodes $t$ on level $2k$ have $|V(T_t)|$ 
neighbours on level $2k+1$ 
			and the graphs assigned to them are all torsos of $(T_t,\Vcal_t)$;
			\item\label{itm_ttd_oddadhesion} if $t\in V(T)$ has level $2k+1$, its 
adhesion is at most~$k$;
			\item\label{itm_ttd_oddneighbours} nodes on level $2k+1$ have at most one 
neighbour on level $2k+2$.
		\end{enumerate}
	\end{thm}

	Note that Theorem~\ref{thm_ttd}\,(\ref{itm_ttd_distinguish}) implies Theorem 
\ref{intro_simple2} and hence Theorem \ref{intro_simple_version}.
	
	\begin{proof}[Proof of Theorem~\ref{thm_ttd}]
	We construct the tree of tree-decompositions recursively. More precisely, we have one step for every node of the rooted tree $(T,r)$, which is constructed recursively during the process. 
	Starting with the root $r$, for each node of $T$ in its step we define a tree-decomposition of its associated graph and define its neighbours at the next level of $T$ and their associated graphs.

	We start by assigning the graph $G$ to the root; that is, we set $G_r:=G$. Assume that a node $t$ of $T$ is defined and we already assigned a graph $G_t$ to this node.
	First we consider the case that the node $t$ is on an odd level. We denote its level by $2k+1$. Let $\Pcal_t$ be the set of profiles induced 
by~$\Pcal$ on~$G_t$.
		If $\Pcal_t$ is empty, let $(T_t,\Vcal_t)$ be the trivial \td\ with a unique node 
and let $t$ have 
		no neighbour on the next level.
		If $\Pcal_t$ is not empty but $k+1\neq \kappa(\Pcal_t,G_t)$ (we note that this always includes the case that the set $\Pcal_t$ consists of a single profile), let $(T_t,\Vcal_t)$ be 
the trivial \td\ 
		of~$G_t$ and let $t$ have a unique neighbour on the next level whose associated 
graph is~$G_t$.
		If $\Pcal_t$ is not empty and $k+1=\kappa(\Pcal_t,G_t)$, let $(T_t,\Vcal_t)$ be the 
canonical \td\ 
		of Proposition~\ref{prop_makeNice}.
		It has adhesion at most $k$ by that proposition.
		Only the unique node of $(T_t,\Vcal_t)$ whose torso contains induced profiles 
of~$\Pcal_t$ has a 
		neighbour on the next level whose associated graph is that torso.
		Note that this torso is well-separable by Proposition~\ref{prop_makeNice}.
		By construction, (\ref{itm_ttd_oddadhesion}) and (\ref{itm_ttd_oddneighbours}) hold.
		
		If $t\in V(T)$ has level $2k$, let $\Pcal_t$ be the set of profiles induced 
by~$\Pcal$ on~$G_t$.
		If $k\neq\kappa(\Pcal,G)$, let $(T^t,\Vcal^t)$ be the trivial \td\ of~$G^t$, 
i.\,e.\ 
let $T^t$ be a 
		tree with one vertex and $\Vcal^t=\{V(G_t)\}$.
		If $k=\kappa(\Pcal,G)$, then $G_t$ is well-separable by construction.
		Let $(T_t,\Vcal_t)$ be the canonical \td\ from 
		Theorem~\ref{thm_canon_td} for~$G_t$ and~$\Pcal_t$.
		Then $t$ gets $|V(T_t)|$ neighbours on the next level whose associated graphs are 
the torsos of 
		$(T_t,\Vcal_t)$.
		Then (\ref{itm_ttd_evenbalanced}) and (\ref{itm_ttd_evenneighbours}) hold for~$t$ 
by 
construction.
This completes the construction of the \ttd\ $((T,r),(G_t)_{t\in 
V(T)},(T_t,\Vcal_t)_{t\in V(T)})$.
By Proposition~\ref{prop_ProfilesInTD}\,(\ref{itm_ProfilesInTD4}), it distinguishes $\Pcal$ efficiently.
	\end{proof}

	As we already mentioned in the introduction, profiles are a generalisation of ends, 
$k$-blocks and 
 tangles in finite graphs.
	As a corollary, it will follow that Theorem~\ref{thm_ttd} holds for the set of profiles 
induced by 
	ends, the set of distinguishable profiles induced by robust $k$-blocks and the set of 
profiles induced 
	by principal tangles of order~$k$.
	Here we will give brief definitions of all these concepts and discuss how they induce 
profiles.
	
	Let $G$ be a graph.
	A \emph{ray} is a one-way infinite path and two rays in~$G$ are \emph{equivalent} if for 
every 
	finite $S\sub V(G)$ there is a component of~$G-S$ such that both rays have all but finitely 
many 
	vertices in that component.
	This is an equivalence relation whose equivalence classes are the \emph{ends} of~$G$.
	Let $\omega$ be an end of~$G$.
	We say that $\omega$ \emph{lies} in a component of $G-S$ if for every ray $R\in\omega$ all 
but 
	finitely many vertices of~$\omega$ lie in $G-S$.
	Let $P_\omega$ be the set of separations $(A,B)$ of finite order such that the component of 
	$G-(A\cap B)$ that contains~$\omega$ lies in~$B$.
	Then (P1) and (P2) are true by definition.
	Obviously, $P_\omega$ is robust and principal.
	Thus, sets of ends define sets of robust profiles.
	
	\begin{coro}\label{cor_ttd1}
		Let $G$ be a graph and let $\Omega$ be a set of ends of~$G$.
		Let $\Pcal$ be the set of profiles defined by~$\Omega$.
		Then there exists a \ttd\ distinguishing all ends in $\Omega$.\qed
	\end{coro}
	
	A \emph{$k$-block} of a graph~$G$ is a maximal set $b$ of at least $k$ vertices such that 
no 
two of 
	its vertices can be separated in~$G$ by fewer than $k$ vertices.
	Then for every separation $(A,B)$ of order at most $k-1$ we have either $b\sub A$ or $b\sub 
B$.
	Let $P_b$ be the set of separations $(A,B)$ of order at most $k-1$ with $b\sub B$.
	It is easy to see that $P_b$ is a principal profile.
	We call $b$ \emph{robust} if $P_b$ is robust and two $k$-blocks are \emph{distinguishable} 
if their 
	profiles are distinguishable.
	
	\begin{coro}\label{cor_ttd2}
		Let $G$ be a graph and let $\Bcal$ be a set of distinguishable robust $k$-blocks.
		Let $\Pcal$ be the set of profiles defined by~$\Bcal$.
		Then there exists a \ttd\ distinguishing all robust profiles in $\Bcal$.\qed
	\end{coro}
	
A \emph{principal tangle of order $k$} in a graph $G$ is a set $\theta$ of separations of order 
at most $k-1$ if it satisfies the following conditions.
\begin{enumerate}[($\theta$1)]
\item For all $(A_1,B_1),(A_2,B_2),(A_3,B_3)\in\theta$, we have 
\[
G\neq G[A_1]\cup G[A_2]\cup G[A_3],
\]
where $G[A_i]$ is the graph induced by the vertex set~$A_i$ for $i\in\{1,2,3\}$;
\item if $X$ is a set of at most $k$ vertices, there is a component $C$ of $G\sm X$ such that 
$(G\sm 
C,C\cup X)\in \theta$;
\item for all separations $(A,B)$ of order at most $k-1$ we have either $(A,B)\in\theta$ or 
$(B,A)\in\theta$.
\end{enumerate}
Two tangles $\theta_1,\theta_2$ are \emph{distinguishable} if there is a separation 
$(A,B)\in\theta_1$ with $(B,A)\in\theta_2$.

\begin{prop}
Let $G$ be a graph.
Every principal tangle of order~$k$ is a robust principal $k$-profile.
\end{prop}

\begin{proof}
Let $\theta$ be a principal tangle and let $(C,D)\leq (A,B)$ with $(A,B)\in \theta$.
If $(D,C)\in \theta$, then $G=G[A]\cup G[D]$ which violates ($\theta$1).
So $\theta$ is consistent.

Let $((A_i, B_i))_{i\in I}$ be a family of separations in~$\theta$ such that $\bigcap_{i\in I} 
B_i\neq \es$.
Let $C$ be a component of $\bigcup A_i\setminus X$, where $X=\bigcap_{i\in I} B_i\cap \bigcup_{i\in 
I} A_i$.
So there is an $A_i$ with $C\subseteq A_i$.
By ($\theta$2) we have $(G\sm C,C\cup X)\in \theta$.
If $(\bigcap_{i\in I} B_i,\bigcup_{i\in I} A_i)\in \theta$, then $(B_i,A_i)\leq (G\sm C,C\cup X)$, 
which contradicts ($\theta$1) for $(A_i,B_i)$ and $(B_i,A_i)$.
Thus, we have $(\bigcap_{i\in I} B_i,\bigcup_{i\in I} A_i)\notin \theta$ and hence, $\theta$ is a 
$k$-profile.
It is principal as the tangle is principal.

To prove robustness of~$\theta$, let $(A,B)\in \theta$ and $(C,D)$ be a separation such that 
$(B\cap 
C,A\cup D)$ and $(B\cap D,A\cup C)$ have order less than~$k$.
If both $(B\cap C,A\cup D)$ and $(B\cap D,A\cup C)$ belong to~$\theta$, then we have
\[
G=G[A]\cup [B\cap C]\cup G[B\cap D]
\]
which is impossible by~($\theta$1).
Thus, $\theta$ is robust.
\end{proof}

	\begin{coro}\label{cor_ttd3}
		Let $G$ be a graph and let $\Pcal$ be a set of distinguishable principal tangles of finite 
order.
		Then there exists a \ttd\ distinguishing all principal tangles of finite order in $\Pcal$.\qed
	\end{coro}

	\section{Locally finite graphs}\label{sec_LocFin}

	In this section, we are applying Theorem~\ref{thm_ttd} to the special case of locally 
finite 
graphs.
	While we will show that for fixed $k\in\N$ there is a canonical \td\ distinguishing all 
	$k$-distinguishable profiles efficiently (Theorem~\ref{thm_LocFinMain1}), it is not 
possible 
to 
	extend this to all distinguishable profiles as Examples~\ref{exam_LocFin} shows.
	But as a further positive result, Theorem~\ref{thm_LocFinMain2} shows that we can at least 
find a 
	nested set of separations distinguishing the distinguishable profiles.
	This nested set does not define a \td\ as it does not satisfy (\ref{property_Finiteness}).
Note that for locally finite graphs, all profiles are principal.
	
	A separation $(A,B)$ of a graph is \emph{tight} if there are components $C_A$ of $A\sm B$ 
and $C_B$ 
	of $B\sm A$ such that every vertex in $A\cap B$ has neighbours in~$C_A$ and in~$C_B$.
	An easy corollary of Lemma~\ref{lem_minimally} is the 
	following.
	
	\begin{prop}\label{prop_corTW}
		Let $G$ be a locally finite graph, let $v\in V(G)$ and let $k\in\N$.
		Then there are only finitely many tight separations of order $k$ with $v$ in their 
separator.\qed
	\end{prop}
	
	For a tree $T$ and be a subset $E$ of~$E(T)$, we denote by $T/E$ the tree obtained by 
contracting 
	all edges of~$E$.
	A \td\ $(T',\Vcal')$ of~$G$ is a \emph{refinement} of a \td\ 
$(T,\Vcal)$ 
	of~$G$ if there is a family of disjoint subtrees $(T_i)_{i\in I}$ of~$T'$ covering $V(T')$ 
such that 
	the following holds:
	\begin{enumerate}[(R1)]
		\item $T=T'/\bigcup_{i\in I} E(T_i)$;
		\item $\bigcup_{s\in T_i} V_s'=V_t$, where $t$ is the node of~$T$ obtained from the 
contraction 
		of~$E(T_i)$.
	\end{enumerate}
	
	If $T$ is a finite tree, then it is well-known that there is either a unique vertex or 
unique edge 
	that lies in the middle of every path of maximum length in~$T$.
	We call this vertex or edge the \emph{central} vertex or edge of~$T$.
	It is preserved by all automorphisms of~$T$.
	
	\begin{prop}\label{prop_TDrefinement}
		Let $G$ be a locally finite graph and $(T,\Vcal)$ be a canonical \td\ of~$G$ of 
finite adhesion.
		For every torso $H_t$ of $(T,V)$ let $(T^t,\Vcal^t)$ be a canonical \td\ of~$H_t$ 
of 
finite adhesion 
		such that every separation $(A,B)$ induced by $(T^t,\Vcal^t)$ is tight and such 
that 
no two adhesion 
		sets induce the same separation.
		Then there is a canonical \td\ $(T',\Vcal')$ that is a refinement of $(T,\Vcal)$ 
with respect to a 
		family $(R^t)_{t\in V(T)}$, where $R^t$ is a subdivision of~$T^t$, such that very 
adhesion set 
		of~$(T',\Vcal')$ is an adhesion set of either $(T,\Vcal)$ or one of the \td s 
$(T^t,\Vcal^t)$.
	\end{prop}
	
	\begin{proof}
		We are going to construct a new \td\ $(T',V')$ of $G$ by gluing together the \td s 
$(T^t,\Vcal^t)$ 
		along the tree $T$ in a canonical way.
		Let $tt'\in E(T)$.
		Let $S^{tt'}$ be the maximal subtree of~$T^t$ such that all $V_s^t$ with $s\in 
V(S^{tt'})$ contain $V_t\cap 
		V_{t'}$.
		Then also all adhesion sets corresponding to edges with both its incident vertices 
in $S^{tt'}$ 
		contain $V_t\cap V_{t'}$.
		As the induced separations of these edges are all distinct and tight, 
Proposition~\ref{prop_corTW} 
		implies that $S^{tt'}$ is finite.
		As we mentioned above, $S^{tt'}$ has a unique central vertex or edge, which is 
fixed 
by all 
		automorphisms of~$S^{tt'}$.
		
		Let $E^t$, $U^t$ be the set of edges, of vertices of~$T^t$ that are a central edge, 
a central 
		vertex, for some tree $S^{tt''}$ with $tt''\in E(T)$, respectively.
		We subdivide all edges in~$E^t$ once and obtain a new tree~$R^t$.
		Let $\Wcal^t$ be a set of vertex sets, one for every $s\in V(R^t)$ such that 
$W_s=V^t_s$ if $s$ is a 
		vertex of~$T^t$ and such that $W_s$ is the adhesion set corresponding to the edge 
$e\in E(T^t)$ if $s$ is the 
		vertex that subdivided~$e$.
		It directly follows from the fact that $(T^t,\Vcal^t)$ is a \td\ that also 
$(R^t,\Wcal^t)$ is one.
		
		Let $T'$ be the graph obtained from the disjoint union of all trees $R^t$ for $t\in 
V(T)$ by adding 
		for every edge $tt'\in E(T)$ an edge between the central vertex or vertex on the 
subdivided central 
		edge of~$S^{tt'}$ and that of~$S^{t't}$.
		It is easy to see that contracting the subgraphs $R^t$ of~$T'$ results in~$T$ and 
hence $T'$ is a 
		tree.
		Let $\Vcal'$ be the union of the sets $\Wcal^t$ for all $t\in V(T)$.
		
		That $(T,\Vcal)$ and all $(R^t,\Wcal_t)$ are \td s implies that $(T',\Vcal')$ is 
one, too.
		The \td\ $(T',\Vcal')$ is canonical as the same is true for $(T,\Vcal)$ and all 
$(R^t,\Wcal^t)$ and 
		by construction of~$T'$.
		As the properties (R1) and (R2) hold by construction, the assertion follows.
	\end{proof}

Now we are able to prove the following, which implies Corollary \ref{intro_loc_fin} already 
mentioned in the Introduction.

	\begin{thm}\label{thm_LocFinMain1}
		Let $G$ be a locally finite graph and let $k\in\N$.
		Let $\Pcal$ be a set of $k$-distinguishable robust profiles each of which is an 
$\ell$-profile for 
		some $\ell\in\N\cup\{\infty\}$.
		Then there is a \td\ that is canonical with respect to~$\Pcal$ and that 
distinguishes $\Pcal$ efficiently.
	\end{thm}
	
	\begin{proof}
		Let $((T,r),(G_t)_{t\in V(T)},(T_t,\Vcal_t)_{t\in V(T)})$ be a \ttd\ of~$G$ with 
the 
properties of 
		Theorem~\ref{thm_ttd}.
		Since $\Pcal$ is $k$-distinguishable, the maximum level of $(T,r)$ is $2k+1$ by 
construction.
		So $2k+1$ iterated applications of Proposition~\ref{prop_TDrefinement}, where in 
each step we use 
		all \td\ of the next level of~$T$, lead to a \td\ $(T',\Vcal)$ of~$G$.
		Since the \ttd\ distinguishes all $k$-distinguishable profiles of~$\Pcal$ 
efficiently, so does 
		$(T',\Vcal)$.
	\end{proof}
	
	In our following example, we will show that we cannot omit the condition 
'$k$-distinguishable' in 
	Theorem~\ref{thm_LocFinMain1}.
	
	\begin{exam}\label{exam_LocFin}
		Let $G'$ be the graph with vertex set 
		\[
		V(G')=\{(n,k/(3^n))|n\in\N, 0\leq k\leq 2\cdot 3^n\}
		\]
		and edges joining $(n,k/(3^n))$ with $(n,(k+1)/(3^n))$ and joining $(n,k/(3^n))$ 
with 
		$(n+1,k/(3^n))$.
		For $k\in\N$ let
		\[
		U_n^1:=\{(n,k/(3^n))\mid 0\leq k\leq 3^n\}\cup \{(n+1,k/(3^{(n+1)}))\mid 0\leq 
k\leq 
3^{(n+1)}\}
		\]
		and
		\begin{align*}
		U_n^2:=&\{(n,k/(3^n))\mid 0\leq k-3^n\leq 3^n\}\\
		&\cup\{(n+1,k/(3^{(n+1)}))\mid 0\leq k-3^{(n+1)}\leq 3^{(n+1)}\}.
		\end{align*}
		Let $\ell_n:=|U_n^1|=|U_n^2|$ and let $G_n^i$, for $i=1,2$, be new graphs with 
vertex set 
		$\{(i,j)\mid i\in\N, 1\leq j\leq \ell_k)\}$ and edges between $(i_1,j_1)$ and 
$(i_2,j_2)$ if 
		$|i_1-i_2|\leq 1$.
		Let $G$ be the disjoint union of $G'$ with all $G^i_n$ for $i=1,2$ and $n\in\N$ 
with 
the following 
		additional edges: for every $n\in\N$ and every $i\in\{1,2\}$ we add an edge between 
every vertex of 
		$U^i_n$ and the vertices $(0,j)$ in~$G^i_n$.
		Note that $G$ is locally finite.
		
		Let $\Pcal$ be the profiles that are induced by the ends of~$G$.
		We will prove that $G$ has no \td\ that distinguishes $\Pcal$ efficiently.
		The end in $G^1_n$ is separated from the end in~$G'$ minimally by the vertex set
		\[
		S_n^1:=\{(n+1,k/(3^{(n+1)}))\mid 0\leq k\leq 3^{(n+1)}\}\cup\{(N,1)\mid 0\leq N\leq 
n\}.
		\]
		Similarly, the end in $G^2_n$ is separated from the end in~$G'$ minimally by the 
vertex set
		\[
		S_n^2:=\{(n+1,k/(3^{(n+1)}))\mid 0\leq k-3^{(n+1)}\leq 3^{(n+1)}\}\cup\{(N,1)\mid 
0\leq N\leq n\}.
		\]
		As $G-S_n^i$ has precisely two components, there is a unique separation 
$(A_n^i,B_m^i)$ with 
		$V(G_n^i)\sub A_n^i$ that distinguishes the profiles corresponding to those ends 
efficiently.
		So if we had a \td\ $(T,\Vcal)$ of~$G$ distinguishing $\Pcal$ efficiently, its 
induced separations 
		are those we just constructed and between $(A_n^1,B_n^1)$ and $(B_n^2,A_n^2)$ there 
are all 
		$(A_N^1,B_N^1)$ and $(B_N^2,A_N^2)$ with $N>n$.
		Hence there are infinitely many edges of~$T$ between the two edges corresponding to 
the separations 
		$(A_n^1,B_n^1)$ and $(B_n^2,A_n^2)$ which is impossible.
		This contradiction shows that no \td\ distinguishes $\Pcal$ efficiently.
	\end{exam}
	
	\begin{thm}\label{thm_LocFinMain2}
		Let $G$ be a locally finite graph and let $\Pcal$ be a set of distinguishable 
robust 
profiles each 
		of which is an $\ell$-profiles for some $\ell\in\N\cup\{\infty\}$.
		Then there is a nested set of separations that is canonical with respect to~$\Pcal$ 
and that distinguishes $\Pcal$ efficiently.
	\end{thm}
	
	\begin{proof}
		Let $((T,r),(G_t)_{t\in V(T)},(T_t,\Vcal_t)_{t\in V(T)})$ be a \ttd\ of~$G$ with 
the 
properties of 
		Theorem~\ref{thm_ttd}.
		For $k\in\N$, let $(T^k,\Vcal^k)$ be the \td\ obtained by applying 
		Proposition~\ref{prop_TDrefinement} iteratively for the the subtree of $(T,r)$ 
consisting of all 
		vertices on the first $2k$ levels.
		By construction, $(T^{k+1},\Vcal^{k+1})$ is a refinement of $(T^k,\Vcal^k)$.
		Let $\Ncal^k$ be the nested set of separations induced by $(T^k,\Vcal^k)$.
		Then $\Ncal^k\sub \Ncal^{k+1}$.
		Thus, $\bigcup_{n\in\N}\Ncal^n$ is a nested set of separations.
		It distinguishes $\Pcal$ efficiently as the \ttd\ does so and it is canonical, as 
all steps in this 
		proof keep this property and the \ttd\ we started with is canonical.
	\end{proof}
	
	\section{Concluding remarks}

In the exposition in the Introduction we focussed on main ideas and gave some theorems in a 
more concrete formulation. Here we summarise some results that are slightly stronger in details 
than those stated in the Introduction. 

	\begin{remark}
	 Theorem \ref{thm_LocFinMain1} is also true if we relax `locally finite' to the property 
that the removal of finitely many vertices only leaves 
finitely many components. The proof is essentially the same. 
	\end{remark}

	\begin{remark}
	 In this paper we proved the strengthening of Theorem \ref{intro_simple2} for arbitrary 
subsets of the set of robust principal profiles, and the tree of \td s we obtain is canonical 
with respect to that subset, compare Theorem \ref{thm_ttd}.
	\end{remark}

	\begin{remark}
	 
		Theorem \ref{thm_ttd} easily implies the following variant. (To see this one has to 
simply `move separations of low order at higher levels down to lower levels'. We leave the details 
to the reader.)

		Let $G$ be a graph and $\Pcal$ a set of distinguishable robust principal profiles each of 
which is an $\ell$-profile for some $\ell\in\N\cup\{\aleph_0\}$.
		Then there exists a \ttd\ $((T,r),(G_t)_{t\in V(T)},(T_t,\Vcal_t)_{t\in V(T)})$ 
that 
is canonical with respect to~$\Pcal$ such that 
		the following hold.
		\begin{enumerate}[{\rm (1)}]
			\item The \ttd\ distinguishes $\Pcal$ efficiently;
			\item if $t\in V(T)$ has level $k$, then $(T_t,\Vcal_t)$ is $k$-balanced;
			\item nodes $t$ at all levels have $|V(T_t)|$ neighbours on the next 
level and the graphs assigned to them are all torsos of $(T_t,\Vcal_t)$.
		\end{enumerate} 
	 \end{remark}

{\bf Further related work.} 
	If we consider the class of quasi-transitive graphs, then Hamann, Lehner, Miraftab and 
R\"uhmann~\cite{HLMR} proved with the aid of our main result that those graphs that admit a 
canonical \td\ distinguishing 
	all their ends are the accessible graphs, that is, the graphs that are obtained from finite 
or one-ended quasi-transitive graphs by tree amalgamations of finite adhesion and finite 
identification respecting the group actions. This result in turn is used in \cite{H-TAQI,H-Asdim, 
H-TAHypBound} to investigate quasi-isometry types, homeomorphism types of hyperbolic boundaries and 
asymptotic dimension of quasi-transitive locally finite graphs.
Also Miraftab and Stavropoulos \cite{miraftab2019splitting} used canonical tree-decompositions to 
classify all infinite groups which admit cubic Cayley
graphs of connectivity 2 in terms of splittings over a subgroup.

\section*{Acknowledgement}

We thank C.~Elbracht and J.~Kneip for pointing out an error in an earlier version of this paper.
	
\providecommand{\bysame}{\leavevmode\hbox to3em{\hrulefill}\thinspace}
\providecommand{\MR}{\relax\ifhmode\unskip\space\fi MR }
\providecommand{\MRhref}[2]{%
  \href{http://www.ams.org/mathscinet-getitem?mr=#1}{#2}
}
\providecommand{\href}[2]{#2}

\end{document}